\documentclass[12pt, reqno]{amsart}

\usepackage{amssymb,amsmath,latexsym}
\usepackage[varg]{pxfonts}
\usepackage{amsthm, amscd, amsfonts, amssymb, graphicx, color}
\usepackage[bookmarksnumbered, colorlinks, plainpages]{hyperref}
\usepackage{marvosym}
\usepackage{paralist}
\usepackage{color}
\theoremstyle{plain}
\newtheorem{lem}{\textbf{Lemma}}[section]
\newtheorem{thm}[lem]{\textbf{Theorem}}

\newtheorem{de}[lem]{\textbf{Definition}}

\newtheorem{ex}[lem]{\textbf{Example}}
\newtheorem{pr}[lem]{\textbf{Proposition}}
\theoremstyle{definition}

\begin{document}

\noindent {}   \\[0.50in]


\title[A New Algorithm for Bregman weak relatively nonexpansive mappings]{A New Shrinking projection Algorithm  for an infinite family of Bregman weak relatively nonexpansive mappings in a Banach Space}

\author[  Orouji, Soori, O'Regan, Agarwal]{Bijan Orouji$^{1,*}$,  Ebrahim Soori$^{2}$,  Donal O'Regan$^{3}$, Ravi P. Agarwal${^{4}}$}
\thanks{$^{*}$ Corresponding author}

\address{$^{1}$Department of Mathematics, Kermanshah Branch, Islamic Azad University, Kermanshah, Iran\\
$^{2}$Department of Mathematics, Lorestan University,  P.O. Box 465, Khoramabad, Lorestan, Iran\\
$^{3}$School of Mathematical and Statistical Sciences,  University of Galway, Galway, Ireland\\
$^{4}$Department of Mathematics, Texas A $\&$ M University Kingsville, Kingsville, USA,}

\email{bijan.orouji@iau.ac.ir(B. Orouji); sori.e@lu.ac.ir(E. Soori);
donal.oregan@nuigalway.ie(D. O'Regan); agarwal@tamuk.edu(R. P. Agarwal)}



\maketitle
{\centering\footnotesize\textbf{This paper is dedicated to the memory of Pro. Wataru Takahashi}.\par}
\hrule width \hsize \kern 1mm

\begin{abstract}
In this paper, using a new shrinking projection method and  generalized resolvents of maximal monotone operators and  generalized projections, we consider the strong convergence for finding a common point of the fixed points of a Bregman quasi-nonexpansive mapping, and common fixed points of a infinite family of Bregman weak relatively nonexpansive mappings, and  common zero points of a finite family of maximal monotone mappings, and  common solutions of an equilibrium problem in a reflexive Banach space.

\end{abstract}
\maketitle
\vspace{0.1in}
\hrule width \hsize \kern 1mm

\section{\textbf{Introduction}}
Throughout this paper, we denote the set of real numbers and the set of positive integers by $\mathbb{R}$ and $\mathbb{N}$, respectively. Let $E$  be a Banach space with the norm $\|.\|$ and the dual space $E^*$. For any $x\in E$, we denote the value of $x\in E^*$ at $x$ by $\langle x, x^*\rangle$. Let $\{x_n\}_{n\in \mathbb{N}}$ be a sequence in $E$. We denote the strong convergence of $\{x_n\}_{n\in \mathbb{N}}$ to $x\in E$ as $n\rightarrow \infty$ by $x_n \rightarrow x$ and the weak convergence by $x_n\rightharpoonup x$. The modulus $\delta$ of convexity of $E$ is denoted by
\begin{align}
\delta_{E}(\varepsilon)=inf \bigg\lbrace 1-\frac{\Vert x+y \Vert}{2}:\Vert x\Vert\leq 1,\Vert y\Vert\leq 1,\Vert x-y\Vert \geq \varepsilon \bigg\rbrace
\end{align}
for all $0\leq\varepsilon\leq 2$. A Banach space $E$ is said to be uniformly convex if $\delta(\varepsilon)>0$ for every $\varepsilon>0$. Let $S_E=\{ x\in E: \|x\|=1\}$. The norm of $E$ is said to be G$\hat{a}$teaux differentiable if for all $x,y\in S_E$, the limit
\begin{align}\label{3yjgh}
\lim _{t_{\rightarrow}0}\frac{\Vert x+ty\Vert-\Vert x\Vert}{t}
\end{align}
exists. In this case, $E$ is called smooth. If the limit \eqref{3yjgh} is attained uniformly for all $x,y\in S_E$, then $E$ is called uniformly smooth. The Banach space $E$ is said to be strictly convex if $ \Vert \frac{x+y}{2} \Vert <1$ whenever $x,y\in S_E$ and $x\neq y$. It is well known that $E$ is uniformly convex if and only if $E^*$ is smooth; for more detalis, see \cite{3qq1,3qq2,3qq3}.

Let $C$ be a nonempty subset of $E$ and $T:C\rightarrow E$ a mapping. We denote the set of fixed points of $T$ by $F(T)$ i.e., $F(T)=\{x\in C: Tx=x\}$. A mapping $T:C\rightarrow E$ is said to be nonexpansive if $\|T(x)-T(y)\|\leq \|x-y\|$ for all $x,y\in C$. A mapping $T:C\rightarrow E$ is said to be quasi-nonexpansive if $F(T)\neq \emptyset$ and $\|T(x)-T(y)\|\leq \|x-y\|$ for all $x\in C$ and $y\in F(T)$. The concept of nonexpansivity plays an important role in the study of Mann-type iteration \cite{3qq4} for finding fixed points of a mapping $T:C\rightarrow C$. Recall that the Mann-type iteration is given by the following formula:
\begin{equation}\label{3ased}
 x_{n+1}=\gamma_nTx_n+(1-\gamma_n)x_n, \qquad x_1\in C.
\end{equation}
Here, $\{\gamma_n\}_{n\in \mathbb{N}}$ is a sequence of real numbers in $[0, 1]$ satisfying some appropriate conditions. The construction of fixed points of nonexpansive mappings via Mann's algorithm \cite{3qq4} has been extensively investigated  in the current literature (see, for example, \cite{3qq5} and the references therein). In \cite{3qq5}, Reich proved that the sequence $\{x_n\}_{n\in \mathbb{N}}$ generated by Mann’s algorithm \eqref{3ased} converges weakly to a fixed point of $T$. However, the convergence of the sequence $\{x_n\}_{n\in \mathbb{N}}$ generated by Mann’s algorithm \eqref{3ased} is not strong in general (see a counter example in \cite{3qq6}; see also \cite{3qq7,3qq8}). Some attempts to modify the Mann iteration method \eqref{3ased} so that strong convergence is guaranteed have appeared. Bauschke and Combettes \cite{3qq9} proposed a modification of the Mann iteration process for a single nonexpansive mapping $T$ in a Hilbert space $H$ and  they proved that if the sequence $\{\alpha_n\}_{n\in \mathbb{N}}$ is bounded above from one, then the sequence $\{x_n\}_{n\in \mathbb{N}}$ generated by \eqref{3ased} converges strongly to a fixed point of $T$; see also Nakajo and Takahashi \cite{3qq10}.

Let $E$ be a smooth, strictly convex and reflexive Banach space and let $J$ be the normalized
duality mapping of $E$. Let $C$ be a nonempty, closed and convex subset of $E$. The generalized
projection $\Pi_C$ from $E$ onto $C$ is defined and denoted by
\begin{equation}\label{3wesd}
  \Pi_C(x)=argmin_{y\in C} \phi(y,x),
\end{equation}
where $\phi(x,y)=\|x\|^2-2\langle x,jy\rangle +\|y\|^2$. For more details, see \cite{3qq11,p5qdhr}.

Let $C$ be a nonempty, closed and convex subset of a smooth Banach space $E$, and let $T$ be a mapping from $C$ into itself. A point $p\in C$ is said to be an asymptotic fixed point \cite{3qq12} of $T$ if there exists a sequence $\{x_n\}_{n\in \mathbb{N}}$ in $C$ which converges weakly to $p$ and $\lim _{n\rightarrow\infty} \|x_n-T(x_n)\|=0$. We denote the set of all asymptotic fixed points of $T$ by $\hat{F}(T)$. A point $p\in C$ is called a strong asymptotic fixed point of $T$ if there exists a sequence $\{x_n\}_{n\in \mathbb{N}}$ in $C$ which converges strongly to $p$ and $\lim _{n\rightarrow\infty} \|x_n-T(x_n)\|=0$. We denote the set of all strong asymptotic fixed points of $T$ by $\tilde{F}(T)$.

Following Matsushita and Takahashi \cite{3qq13}, a mapping $T:C\rightarrow C$ is said to be relatively nonexpansive if the following conditions are satisfied:
\begin{itemize}
  \item [{\rm(i)}] $F(T)$ is nonempty,
  \item [{\rm(ii)}] $\phi(p,Tx)\leq \phi(p,x),\quad\forall x\in C, p\in F(T)$,
  \item [{\rm(iii)}]$\hat{F}(T)=F(T)$.
  \end{itemize}

In 2005, Matsushita and Takahashi \cite{3qq13} proved the following strong convergence theorem for relatively nonexpansive mappings in a Banach space.
\begin{thm}\label{3bnfht}
 Let $E$ be a uniformly convex and uniformly smooth Banach space, let $C$ be
a nonempty, closed and convex subset of $E$, let $T$ be a relatively nonexpansive mapping
from $C$ into itself, and let $\{\alpha_n\}_{n\in \mathbb{N}}$ be a sequence of real numbers such that $0\leq \alpha_n<1$ and $\lim sup _{t\rightarrow\infty} \alpha_n<1$. Suppose that $\{x_n\}_{n\in \mathbb{N}}$ is given by
 \begin{align}\label{mnasd}
  \begin{cases}
  x_0=x\in C,\\
  y_n=J_{E}^{-1}\Big((1-\alpha_n)J_{E}Tx_n+\alpha_nJ_{E}x_n\Big),\\
  H_n =\lbrace z\in C:\phi(z,y_n)\leq \phi(z,x_n)\rbrace,\\
  W_n =\lbrace z\in C:\langle x_n-z,J_{E}x-J_{E}x_n\rangle\geq 0\rbrace,\\
  x_{n+1}=\Pi_{H_n\cap W_n}x,\quad \forall n\in \mathbb{N}.
  \end{cases}
  \end{align}
If $F(T)$ is nonempty, then $\{x_n\}_{n\in \mathbb{N}}$ converges strongly to $\Pi_{F(T)}x$.
\end{thm}

Also, in 2015,  Naraghirad and  Timnak \cite{p19a} extend Theorem \ref{3bnfht} and proved the following strong convergence theorem for an infinite  family of Bregman weak relatively nonexpansive mappings (to be defined in Section 2)  in a Banach space.
\begin{thm}
Let $E$ be a reflexive Banach space and $g:E\rightarrow \mathbb{R}$ be a convex, continuous, strongly coercive and G$\hat{a}$teaux differentiable function which is bounded on bounded subsets and uniformly convex on bounded subsets of $E$. Let $C$ be a nonempty, closed and convex subset of $E$. Let $\{S_n\}_{n\in \mathbb{N}}$ be a family of Bregman weak relatively nonexpansive mappings of $C$ into itself such that $F:=\cap_{i=1}^\infty F(S_n)\neq \emptyset$, and  let $\{\beta_{n,k}:k,n\in \mathbb{N}, 1\leq k\leq n\}$ be a sequence of real numbers such that $0<\beta_{n,1}\leq 1$ and $0<\beta_{n,i}<1$ for $n\in \mathbb{N}$ and every $i=2, 3, ..., n$. Let $W_n$ be the Bregman $W$-mapping generated by $S_n, S_{n-1}, ..., S_1$ and  $\beta_{n,n}, \beta_{n,n-1}, ..., \beta_{n,1}$. Let $\{\alpha_n\}_{n\in \mathbb{N}}$ be a sequence in $[0, 1)$ such that $\lim_{n\rightarrow +\infty}\alpha_n(1-\alpha_n)>0$. Let $\{x_n\}_{n\in \mathbb{N}}$ be a sequence generated by
 \begin{align}\label{sdomrt}
  \begin{cases}
  x_0=x\in C \qquad chosen \:\: arbitrarily,\\
  y_n=\nabla g^*\Big((1-\alpha_n)\nabla g(W_n(x_n))+\alpha_n\nabla g(x_n)\Big),\\
  C_n =\lbrace z\in C:\phi(z,y_n)\leq \phi(z,x_n)\rbrace,\\
  x_{n+1}=Proj^g_{C_{n+1}}x,\quad \forall n\in \mathbb{N}\cup \{0\}.
  \end{cases}
  \end{align}
  Then $\{x_n\}_{n\in \mathbb{N}}$ converges strongly to $proj^g_F(x_0)$ as $n\rightarrow \infty$.
\end{thm}

In this paper, using  maximal monotone operators, and a shrinking projection method we consider the strong convergence for finding a common point of the fixed points of a Bregman quasi-nonexpansive mapping, and common fixed points of a infinite family of Bregman weak relatively nonexpansive mappings, and  common zero points of a finite family of maximal monotone mappings, and  common solutions of an equilibrium problem in a reflexive Banach space.

\section{\textbf{Preliminaries}}
Let $E$ be a reflexive real Banach space and $C$ be a nonempty closed and convex subset of $E$. Suppose that $h:C\times C\rightarrow \mathbb{R}$ is a bifunction which  satisfies the condition $h(x, x)=0$ for every $x\in C$. The equilibrium problem for $h$ is stated as follows:
\begin{equation}\label{khityn}
  \text{Find} \quad x^*\in C\quad \text{such that} \quad h(x^*, y)\geq 0,\quad \forall y\in C.
\end{equation}
The set of solutions of \eqref{khityn} is denoted by $EP(h)$.

Suppose that $f:E\rightarrow(-\infty,+\infty]$ is a proper, lower semicontinuous and convex function. Let $dom f:=\{x\in E;\;f(x)<\infty\}$, the domain of $f$.  The function $f$ is said to be strongly coercive if $\displaystyle\lim_{\|x\|\rightarrow\infty}\frac{f(x)}{\|x\|}=+\infty$. Let $x\in$ \textit{int\,dom}$f$, and the subdifferential of $f$ at $x$ is the convex set valued mapping $\partial f:E\rightarrow 2^{E^*}$ defined by
\begin{equation*}
\partial f(x)=\{\xi\in E^*:\;f(x)+\langle y-x,\xi\rangle\leq f(y),\;\forall y\in E\},\quad \forall x\in E,
\end{equation*}
 and $f^*:E^*\rightarrow (-\infty,+\infty]$ is the Fenchel conjugate of $f$ defined by
\begin{align*}
  f^*(\xi)=sup\{\langle \xi,x\rangle-f(x):x\in E\}.
\end{align*}
It is well known that $\xi \in \partial f(x)$ is equivalent to
\begin{align*}
f(x)+ f^*(\xi)=\langle x, \xi\rangle.
\end{align*}
 Note $f^*$ is a proper convex and lower semicontinuous function. The function $f$ is said to be cofinite if \textit{dom}$f^*=E^*$.

For any $x\in$ \textit{int\,dom}$f$ and $y\in E$, we denote by $f^\circ(x,y)$, the right-hand dervative of $f$ at $x$ in the direction $y$, that is,
\begin{equation}\label{solkf}
f^\circ(x,y):=\displaystyle\lim_{t\rightarrow0^+}\frac{f(x+ty)-f(x)}{t}.
\end{equation}
The function $f$ is called G$\hat{\text{a}}$teaux differentiable at $x$, if the limit in \eqref{solkf}
 exists for any $y\in E$. In this case, the gradient of $f$ at $x$ is the linear function $\nabla f$ which is defined by $\langle y,\nabla f(x)\rangle:=f^\circ(x,y)$ for any $y\in E$. The function $f$ is said to be G$\hat{\text{a}}$teaux differentiable if it is G$\hat{\text{a}}$teaux differentiable at each $x\in int\,dom f$. The function $f$ is said to be Fr$\acute{\text{e}}$chet differentiable at $x$, if the limit in \eqref{solkf} is attained uniformly in $\|y\|=1$, for any $y\in E$. Finally, $f$ is said to be uniformly Fr$\acute{\text{e}}$chet differentiable on a subset $C$ of $E$, if the limit in \eqref{solkf} is attained uniformly for $x\in C$ and $\|y\|=1$.
\begin{lem}\cite{p24qa}\label{zakhv}
  If $f:E\rightarrow \mathbb{R}$ is uniformly $Fr\acute{e}chet$ differentiable and bounded on bounded subsets of $E$, then $f$ is uniformly continuous on bounded subsets of $E$ and $\nabla f$ is uniformly continuous on bounded subsets of $E$ from the strong topology of $E$ to the strong topology of $E^*$.
\end{lem}
\begin{pr}\cite{p31}
 Let $f:E\rightarrow \mathbb{R}$ be a convex function which is bounded on bounded subsets of $E$. Then the following assertions are equivalent:
 \begin{itemize}
  \item [{\rm(i)}] $f$ is strongly coercive and uniformly convex on bounded subsets of $E$.
  \item [{\rm(ii)}] $f^*$ is $Fr\acute{e}chet$ differentiable and $\triangledown f^*$ is uniformly norm-to-norm continuous on bounded subsets of $domf^*=E^*$.
  \end{itemize}
\end{pr}

The function $f$ is said to be Legendre if it satisfies the following two conditions:\\
  (L1) $int\,dom f\neq \emptyset$ and $\partial f$ is single-valued on its domain.\\
  (L2) $int\,dom f^*\neq \emptyset$ and $\partial f^*$ is single-valued on its domain.\\
Because here the space $E$ is assumed to be reflexive, we always have $(\partial f)^{-1}=\partial f^*$ \cite[p. 83]{p6}. This fact, when combined with the conditions (L1) and (L2), implies the following equalities:
\begin{align*}
\nabla f&=(\nabla f^*)^{-1},\\
ran \nabla f&=dom \nabla f^*=int\,dom f^*,\\
ran \nabla f^*&=dom \nabla f=int\,dom f.
\end{align*}
In addition, the conditions  (L1) and (L2), in conjunction with Theorem 5.4 of \cite{p2}, imply that the functions $f$ and $f^*$ are strictly convex on the interior of their respective domains and $f$ is Legendre if and only if $f^*$ is Legendre.
One important and interesting Legendre function is $\frac{1}{p}\|.\|^p,\;p\in (1,2]$.

 When $E$ is a uniformly convex and $p$-uniformly smooth Banach space with $p\in (1,2]$, the generalized duality mapping $J_p:E\rightarrow2^{E^*}$ is defined by
\begin{equation*}
J_p(x)=\{j_p(x)\in E^*:\langle j_p(x),x\rangle=\|x\|.\|j_p(x)\|\;,\;\|j_p(x)\|=\|x\|^{p-1}\}.
\end{equation*}
In this case, the gradient $\nabla f$ of $f$  coincides with the generalized duality mapping $J_p$ of $E$, $\nabla f=J_p,\;p\in(1,2]$.
Several interesting examples of Legendre functions are presented in \cite{p2,p3,p4,p201}.
\begin{de}\cite{p14}
Let $f:E\rightarrow(-\infty,+\infty]$ be a convex and $G\hat{a}teaux$ differentiable function.The bifunction $D_f:domf\times int\;domf\rightarrow [0,+\infty)$ defined by
\begin{align*}
D_f(y,x):=f(y)-f(x)-\langle \nabla f(x),y-x\rangle.
\end{align*}
is called the Bregman distance with respect to f.
\end{de}
It should be noted that $D_f$ is not a distance in the usual sense of the term. Clearly, $D_f(x,x)=0$, but $D_f(y,x)=0$ may not imply $x=y$. In our case, when $f$ is Legendre this indeed holds \cite[Theorem 7.3(vi), p.642]{p2}. In general, $D_f$  satisfies the three point identity
\begin{equation}\label{awpon}
D_f(x,y)+D_f(y,z)-D_f(x,z)=\langle x-y, \nabla f(z)-\nabla f(y)\rangle,
\end{equation}
and the four point identity
\begin{equation*}
D_f(x,y)+D_f(\omega,z)-D_f(x,z)-D_f(\omega,y)=\langle x-\omega,\nabla f(z)-\nabla f(y)\rangle,
\end{equation*}
for any $x,\omega \in domf$ and $y,z\in int\,dom f$. Over the last 30 years, Bregman distances have been studied by many researchers (see \cite{p2,p5,p10,p11,p202}).

Let $f:E\rightarrow(-\infty,+\infty]$ be a convex function on E which is G$\hat{\text{a}}$teaux differentiable on \textit{int dom}$f$. The function $f$ is said to be totally convex at a point $x\in int\,dom\,f$ if its modulus of total convexity at $x,\;v_f(x,.):[0,+\infty)\rightarrow [0,+\infty]$, defined by
\begin{align*}
 v_f(x,t)=inf \{D_f(y,x):y\in dom f,\|y-x\|=t\},
\end{align*}
is positive whenever $t>0$. The function $f$ is said to be totally convex when it is totally convex at every point of \textit{int dom}$f$. The function f is said to be totally convex on bounded sets, if for any nonempty bounded set $B\subseteq E$ , the modulus of total convexity of $f$ on $B$, $v_f(B,t)$ is positive for any $t>0$, where $v_f(B,.):[0,+\infty)\rightarrow [0,+\infty]$  is defined by
\begin{align*}
  v_f(B,t)=inf\{v_f(x,t):x\in B\cap int\,dom\,f\}.
\end{align*}
We remark in passing that $f$ is totally convex on bounded sets if and only if $f$ is uniformly convex on bounded sets; (see \cite{p5qsd,p11.1,p12}).
\begin{pr}\cite{p11}\label{aplenfg}
 Let $f$ be a lower semicontinuous convex function with \textit{int dom}$f\neq \emptyset$. Then the function $f$ is differentiable at the point $x\in \textit{int dom}f$ if and only if $\partial f(x)$ consists of a single element.
\end{pr}
\begin{pr}\cite{p11.1}
 Let $f:E\rightarrow(-\infty,+\infty]$ be a convex function and that it's domain contains at least two points. If $f$ is lower semicontinuous, then $f$ is totally convex on bounded sets if and only if $f$ is uniformly convex on bounded sets.
\end{pr}
\begin{lem}\cite{p21}
If $x\in int\;dom f$, then the following statements are equivalent:
\begin{itemize}
  \item [{\rm(i)}] The function f is totally convex at x.
  \item [{\rm(ii)}] For any sequence $\{y_n\}\subset\;dom f$,
  \end{itemize}
\begin{equation*}
\displaystyle\lim_{n\rightarrow +\infty}D_f(y_n,x)=0\Rightarrow \displaystyle\lim_{n\rightarrow +\infty}\|y_n-x\|=0.
\end{equation*}
\end{lem}
Recall that the function $f$ is called sequentially consistent \cite{p12}, if for any two sequences $\{x_n\}_{n\in \mathbb{N}}$ and $\{y_n\}_{n\in \mathbb{N}}$ in $E$ such that $\{x_n\}_{n\in \mathbb{N}}$ is bounded, then
\begin{equation*}
\displaystyle\lim_{n\rightarrow +\infty}D_f(y_n,x_n)=0\Rightarrow \displaystyle\lim_{n\rightarrow +\infty}\|y_n-x_n\|=0.
\end{equation*}
\begin{lem}\label{qpocv} \cite{p11}
If dom f contains at least two points, then the function f is totally convex on bounded sets if and only if the function f is sequentially consistent.
\end{lem}
\begin{lem}\cite{p22}\label{qwpo}
Let $f:E\rightarrow \mathbb{R}$ be a $G\hat{a}teaux$ differentiable and totally convex function. If $x_1\in E$ and the sequence $\{D_f(x_n,x_1)\}$ is bounded, then the sequence $\{x_n\}$ is also bounded.
\end{lem}
\begin{lem}\cite{p24}\label{fdyuq}
Let $f:E\rightarrow \mathbb{R}$ be a Legendre function such that $\nabla f^*$ is bounded on bounded subsets of $int\,dom f^*$. Let $x_1\in E$ and if $\{D_f(x_1,x_n)\}$ is bounded, then the sequence $\{x_n\}$ is bounded too.
\end{lem}
Recall that the Bregman projection \cite{p8} with respect to $f$ of $x\in int\,dom f$ onto a nonempty, closed and convex set $C\subseteq int\,dom f$ is the unique vector $proj_C^f(x)\in C$ satisfying
\begin{equation*}
D_f\big(proj_C^f(x),x\big)=inf\{D_f(y,x):y\in C\}.
\end{equation*}
Similar to the metric projection in Hilbert spaces, the Bregman projection with respect to totally convex and G$\hat{\text{a}}$teaux differentiable functions has a variational characterization \cite[corollary 4.4, p. 23]{p12}.
\begin{lem}\label{rtmnw}\cite{p12}
Suppose that f is $G\hat{a}teaux$ differentiable and totally convex on $int\,dom f$. Let $x\in int\,dom f$ and $C\subseteq int\,dom f$ be a nonempty, closed and convex set. Then the following Bregamn projection conditions are equivalent:
\begin{itemize}
  \item [{\rm(i)}] $z_0=proj_C^f(x)$.
  \item [{\rm(ii)}] $z=z_0$ is the unique solution of the following variational inequality
\begin{equation*}
\langle z-y,\nabla f(x)-\nabla f(z)\rangle \geq0,\quad \forall y\in C.
\end{equation*}
\item [{\rm(iii)}]  $z=z_0$ is the unique solution of the following variational inequality
\begin{equation*}
D_f(y,z)+D_f(z,x)\leq D_f(y,x),\quad \forall y\in C.
\end{equation*}
\end{itemize}
\end{lem}
\begin{lem}\label{khecpo}\cite{3qq40}
  Let $C$ be a nonempty convex subset of $E$ and $f:C\rightarrow \mathbb{R}$ be a convex and subdifferentiable function on $C$. Then, $f$ attains its minimum at $x\in C$ if and only if $0\in \partial f(x) + N_C(x)$, where $N_C(x)$ is the normal cone of $C$ at $x$, that is
  \begin{equation*}
    N_C(x) := \{x^*\in E^* :\langle x-z, x^*\rangle \geq 0, \forall z\in C\} .
  \end{equation*}
\end{lem}
  \begin{lem}\label{ehfbqw}\cite{p14.1}
    If $f$ and $g$ are two convex functions on $E$, such that there is a point $x_0\in \text{dom}\:f \cap \text{dom}\:g$ where $f$ is continuous, then
    \begin{equation*}
      \partial(f+g)(x)=\partial f(x)+\partial g(x),\quad \forall x\in E.
    \end{equation*}
  \end{lem}
As before a function $h:C\times C\rightarrow (-\infty, +\infty]$, where $C\subseteq X$ is a closed convex
subset, such that $h(x,x)=0$ for all $x\in C$ is called a bifunction. Throughout this paper, we consider bifunctions with the following properties:\\
A1. $h$ is pseudomonotone, i.e., for all $x,y\in C$:
\begin{equation*}
  h(x,y)\geq 0 \Rightarrow h(y,x)\leq 0.
\end{equation*}
A2. $h$ is Bregman–Lipschitz-type continuous, i.e., there exist two positive
constants $c_1, c_2$, such that
\begin{equation*}
  h(x,y)+h(y,z)\geq h(x,z)-c_1D_f(y,x)-c_2D_f(z,y),\quad \forall x,y,z \in C,
\end{equation*}
where $f: X\rightarrow (-\infty, +\infty]$ is a Legendre function. The constants $c_1, c_2$ are called Bregman-Lipschitz coefficients with respect to $f$.\\
A3. $h$ is weakly continuous on $C\times C$, i.e. if $x,y\in C$ and $\{x_n\}$ and $\{y_n\}$ are two sequences in $C$ converging weakly to $x$ and $y$, respectively, then $h(x_n,y_n)\rightarrow h(x,y)$.\\
A4. $h(x,.)$ is convex, lower semicontinuous and subdifferentiable on $C$ for every fixed $x\in C$.\\
A5. for each $x, y, z \in C$, $\limsup_{t\downarrow 0} h(tx+(1-t)y,z)\leq h(y,z)$.

A bifunction $h$ is called monotone on $C$ if for all $x, y\in C$, $h(x, y) + h(y, x)\leq 0$. It is clear that any monotone bifunction is a pseudomonotone, but not vice versa. A mapping $A: C \rightarrow X^*$ is pseudomonotone if and only if the bifunction $h(x, y)=\langle A(x), y-x\rangle$ is pseudomonotone on $C$ (see \cite{p30qq}).
\begin{ex}
  Let $C\subseteq\mathbb{R}$, $k\leq -1$ and $f:\mathbb{R} \rightarrow \mathbb{R}$ be a Legendre function.
Define $h:C\times C\rightarrow \mathbb{R}$ as follows:
\begin{equation*}
  h(x,y)=|x-y|+k(x-y),\quad \forall x,y\in C.
\end{equation*}
\end{ex}
It is easy to check that the bifunction $h$ satisfies the conditions A1–A5.\\
Associated with the primal form (EP), its dual form is defined as follows \cite{p5qq,p16qq}:
\begin{equation}\label{sdpou}
  \text{Find}\quad y^*\in C \quad\text{such that}\quad h(x,y^*)\leq 0, \quad \forall x\in C.
\end{equation}
Let us denote by DEP($h$) the solution sets of Problem \eqref{sdpou}.
\begin{lem}\label{aspoxc}\cite{p14.52}
  If a bifunction h satisfies the conditions A1, A3-A5, then EP(h)=DEP(h) and EP(h) are closed and convex.
\end{lem}
\begin{lem}\label{demnsd}\cite{p14.52}
Let $f:E\rightarrow \mathbb{R}$ be a convex, strongly coercive, lower semicontinuous, G$\hat{a}$teaux differentiable, and cofinite function. Let $h:C\times C\rightarrow (-\infty, \infty]$ be a function, such that $h(x, .)$ is proper, convex, and lower semicontinuous on $C$ for every fixed $x\in C$. Then, for every $\alpha \in(0,+\infty)$ and $x\in C$, there exists $z\in C$, such that
\begin{equation*}
z\in Argmin \big\{\alpha h(x_n,u)+D_f(u, x_n): u\in C\big\}
\end{equation*}
Furthermore, if $f$ is strictly convex, then this point is unique.
\end{lem}
\begin{lem}\label{swdczg}\cite{p14.52}
  Let $C$ be a nonempty closed convex subset of a reflexive Banach space $E$, and $f:E\rightarrow \mathbb{R}$ be a Legendre and strongly coercive function. Suppose that $h:C\times C\rightarrow \mathbb{R}$ be a bifunction satisfying A1–A4. For the arbitrary sequences $\{x_n\}\subseteq C$ and $\{\alpha_n\}\subseteq (0,+\infty)$, let $\{w_n\}$ and $\{z_n\}$ be sequences generated by
  \begin{align}\label{qpvbgh}
  \begin{cases}
    t_n= argmin \big\{\alpha_nh(x_n,u)+D_f(u, x_n): u\in C\big\},\\
    v_n= argmin \big\{\alpha_nh(t_n,u)+D_f(u, x_n): u\in C\big\}.
  \end{cases}
  \end{align}
  Then, for all $x^*\in EP(h)$
  \begin{equation*}
    D_f(x^*,v_n)\leq D_f(x^*,x_n)-(1-\alpha_n c_1)D_f(t_n,x_n)-(1-\alpha_n c_2)D_f(v_n,t_n).
  \end{equation*}
\end{lem}
Let $E$ be a Banach space and let $B_r=\{z\in E:\|z\|\leq r\}$ for all $r >0$. Then a function $g: E \rightarrow
\mathbb{R}$ is said to be uniformly convex on bounded subsets of $E$ \cite{p31} if $\rho_r(t)>0$ for all $r,t>0$, where $\rho_r:[0,+\infty)\rightarrow [0,+\infty]$ is defined by
\begin{equation*}
 \rho_r(t)=\displaystyle\inf_{x,y\in B_r,\|x-y\|=t,\alpha \in (0,1)}\frac{\alpha g(x)+(1-\alpha)g(y)-g(\alpha x+(1-\alpha)y)}{\alpha(1-\alpha)}
\end{equation*}
for all $t\geq0$. The function $\rho_r$ is called the gauge of uniform convexity of $g$.

Let $f: E \rightarrow(-\infty,+\infty]$ be a proper, lower semicontinuous and convex function. Let $C$ be a nonempty, closed and convex subset of \textit{int\,dom} $f$ and $T:C\rightarrow C$ be a mapping. Now $T$ is said to be Bregman quasi-nonexpansive, if the following conditions are satisfied:
  \begin{itemize}
  \item [{\rm(i)}] $F(T)$ is nonempty,
  \item [{\rm(ii)}] $D_f(p,Tx)\leq D_f(p,x),\quad \forall x\in C, p\in F(T)$.
  \end{itemize}

A mapping $T:C\rightarrow C$ is said to to be Bregman relatively nonexpansive if the following conditions are satisfied:
  \begin{itemize}
  \item [{\rm(i)}] $F(T)$ is nonempty,
  \item [{\rm(ii)}] $D_f(p,Tx)\leq D_f(p,x),\quad\forall x\in C, p\in F(T)$,
  \item [{\rm(iii)}] $\hat{F}(T)=F(T)$.
  \end{itemize}

A mapping $T:C\rightarrow C$ is said to to be Bregman weak relatively nonexpansive if the following conditions are satisfied:
\begin{itemize}
  \item [{\rm(i)}] $F(T)$ is nonempty,
  \item [{\rm(ii)}] $D_f(p,Tx)\leq D_f(p,x),\quad\forall x\in C, p\in F(T)$,
  \item [{\rm(iii)}] $\tilde{F}(T)=F(T)$.
  \end{itemize}
  It is clear that any Bregman relatively nonexpansive mapping is a Bregman quasi nonexpansive mapping. It is also obvious that every Bregman relatively nonexpansive mapping is a Bregman weak relatively nonexpansive mapping, but the converse in not
true in general; see, for example, \cite{3qq15}. Indeed, for any mapping $T : C\rightarrow C$, we have
$F(T)\subset \tilde{F}(T)\subset \hat{F}(T)$. If $T$ is Bregman relatively nonexpansive, then $F(T)= \tilde{F}(T)=\hat{F}(T)$.

Let $C$ be a nonempty, closed and convex subset of \textit{int dom} $f$. An operator $T:C\rightarrow int\,dom f$ is said to be Bregman strongly nonexpansive with respect to a nonempty $\hat{F}(T)$, if
\begin{equation*}
D_f(y,Tx)\leq D_f(y,x),\quad \forall x\in C,y\in \hat{F}(T),
\end{equation*}
and for any bounded sequence $\{x_n\}\subseteq C$ with
\begin{equation*}
\displaystyle\lim_{n\rightarrow \infty}\big(D_f(y,x_n)-D_f(y,Tx_n)\big)=0,
\end{equation*}
it follows that
\begin{equation*}
\displaystyle\lim_{n\rightarrow \infty}D_f(Tx_n,x_n)=0.
\end{equation*}

A mapping $B:E\rightarrow2^{E^*}$ is called Bregman inverse strongly monotone on the set $C$, if $C\cap (int\,dom f)\neq \emptyset$ and for any $x,y\in C\cap (int\,dom f),\xi\in Bx$ and $\eta \in By$, we have that
\begin{equation*}
\big\langle \xi-\eta,\;\nabla f^*(\nabla f(x)-\xi)-\nabla f^*(\nabla f(y)-\eta)\big\rangle \geq0.
\end{equation*}

Let $B:E\rightarrow 2^{E^*}$ be a mapping. Then the mapping defined by
\begin{equation*}
B_\lambda ^f:=\nabla f^*\circ(\nabla f-\lambda B):E\rightarrow E
\end{equation*}
is called an anti-resolvent associated with $B$ and $\lambda$ for any $\lambda>0$.

Suppose that $A$ is a mapping of $E$ into $2^{E^{*}}$ for the real reflexive Banach space $E$. The effective domain of $A$ is denoted by $dom(A)$, that is, $dom(A)=\lbrace x\in E: Ax\neq \emptyset \rbrace$. A multi-valued mapping $A$ on $E$ is said to be monotone if $\langle x-y, u^{*}-v^{*}\rangle\geq 0 $ for all $ x,y \in dom(A), u^{*} \in Ax$ and $v^{*}\in Ay$. A monotone operator $A$ on $E$ is said to be maximal if graph $A$, the graph of $A$, is not a proper subset of the graph of any monotone operator on $E$.

Let $E$ be a real reflexive Banach space, $f:E\rightarrow (-\infty,+\infty]$ be a uniformly Fr$\acute{\text{e}}$chet differentiable and bounded on bounded subsets of $E$, then for any $\lambda >0$ the resolvent of $A$   defined by
\begin{equation*}
Res_A^f(x)=(\nabla f+\lambda A)^{-1}\circ \nabla f(x).
\end{equation*}
is a single-valued Bregman quasi-nonexpansive mapping from $E$ onto $dom(A)$ and $F(Res_A^f)=A^{-1}0$. We denote by $A_\lambda=\frac{1}{\lambda}(\nabla f-\nabla f(Res_A^f))$ the Yosida approximation of $A$ for any $\lambda>0$. We get from \cite[ prop 2.7, p.10]{p22} that
\begin{equation*}
A_\lambda(x)\in A\big(Res_A^f(x)\big),\:\forall x\in E,\;\lambda>0,
\end{equation*}
 (see \cite{p21}, too).
\begin{lem}\label{qctas}\cite{p23}
Let $E$ be a real reflexive Banach space and $f:E\rightarrow(-\infty,+\infty]$ be a Legendre function which is totally convex on bounded subsets of $E$. Also let $C$ be a nonempty closed and convex subset of $int\,dom f$ and $T:C\rightarrow 2^C$ be a multi valued Bregman quasi-nonexpansive mapping. Then the fixed point set $F(T)$ of $T$ is a closed and convex subset of $C$.
\end{lem}
\begin{lem}\label{ewca}\cite{p16}
Assume that $f:E\rightarrow \mathbb{R}$ is a Legendre function which is uniformly $Fr\acute{e}chet$ differentiable and bounded on bounded subsets of $E$. Let $C$ be a nonempty closed and convex subset of $E$. Also let $\{T_i:i=1, ..., N\}$ be $N$ Bregman strongly nonexpansive mappings which satisfy $\hat{F}(T_i)=F(T_i)$ for each $1\leq i\leq N$ and let $T=T_NT_{N-1}...T_1$. If $F(T)$ and $\bigcap_{i=1}^NF(T_i)$ are nonempty, then $T$ is also Bregman strongly nonexpansive with $F(T)=\hat{F}(T)$.
\end{lem}
\begin{lem}\label{vtye}\cite{p20}
Let $G:E\rightarrow 2^{E^*}$ be a maximal monotone operator and $B:E\rightarrow E^*$ be a Bregman inverse strongly monotone mapping such that $(G+B)^{-1}(0^*)\neq \emptyset$. Also let $f:E\rightarrow \mathbb{R}$ be a Legendre function which is uniformly $Fr\acute{e}chet$ differentiable and bounded on bounded subset of $E$. Then
\begin{itemize}
  \item [{\rm(i)}] $(G+B)^{-1}(0^*)=F(Res_{\lambda G}^f\circ B_\lambda ^f)$.
  \item [{\rm(ii)}]  $Res_{\lambda G}^f\circ B_\lambda ^f$ is a Bregman strongly nonexpansive mapping such that
  \begin{equation*}
F(Res_{\lambda G}^f\circ B_\lambda ^f)=\hat{F}(Res_{\lambda G}^f\circ B_\lambda ^f).
  \end{equation*}
  \item [{\rm(iii)}] $ D_f\big(u,Res_{\lambda G}^f\circ B_\lambda ^f(x)\big)+D_f\big(Res_{\lambda G}^f\circ B_\lambda ^f(x),x\big)\leq D_f(u,x),\;\forall u\in(G+B)^{-1}(0^*),\;x\in E\;and\;\lambda>0$.
  \end{itemize}
\end{lem}
\begin{lem}\label{awqqq}\cite{p26}
Let $f:E\rightarrow (-\infty,+\infty]$ be a proper convex and lower semicontinuous Legendre function. Then for any $z\in E$, for any $\{x_n\}\subseteq E$ and $\{t_i\}_{i=1}^N\subseteq (0,1)$ with $\sum_{i=1}^Nt_i=1$, the following holds:
\begin{equation}\label{admwr}
D_f\bigg(z,\nabla f^*\big(\sum_{i=1}^Nt_i\nabla f(x_i)\big)\bigg)\leq \sum_{i=1}^Nt_iD_f(z,x_i).
\end{equation}
\end{lem}
\begin{pr}\label{awlzx}(\cite{p22} prop.2.8,p10)
Let $f$ be a $G\hat{a}teaux$ differentiable and $A:E\rightarrow 2^{E^*}$ be a maximal monotone operator such that $A^{-1}0\neq \emptyset$. Then
\begin{equation*}
D_f(q,x)\geq D_f\big(q,Res_{rA}^f(x)\big)+D_f\big(Res_{rA}^f(x),x\big).
\end{equation*}
for all $r>0,\;q\in A^{-1}0$ and $x\in E$.
\end{pr}
\begin{de}\cite{p19a}
Let $C$ be a nonempty, closed and convex subset of a reflexive Banach space $E$. Let $\{T_n\}_{n\in \mathbb{N}}$ be an infinite family of Bregman weak relatively nonexpansive mappings of $C$ into itself, and let $\{\beta_{n,k}:k,n\in \mathbb{N} , 1\leq k\leq n\}$ be a sequence of real numbers such that $0\leq\beta_{i,j}\leq1$ for every $i,j\in \mathbb{N}$ with $i\geq j$. Then, for any $n\in \mathbb{N}$, we define mapping $W_n$ of $C$ into itself as follows:
\begin{align*}
U_{n,n+1}x&=x,\\
U_{n,n}x&=proj_C^f\big(\nabla f^*[\beta_{n,n}\nabla f(T_nU_{n,n+1}x)+(1-\beta_{n,n})\nabla f(x)]\big),\\
U_{n,n-1}x&=proj_C^f\big(\nabla f^*[\beta_{n,n-1}\nabla f(T_{n-1}U_{n,n}x)+(1-\beta_{n,n-1})\nabla f(x)]\big),\\
\vdots\\
U_{n,k}x&=proj_C^f\big(\nabla f^*[\beta_{n,k}\nabla f(T_kU_{n,k+1}x)+(1-\beta_{n,k})\nabla f(x)]\big),\\
\vdots\\
U_{n,2}x&=proj_C^f\big(\nabla f^*[\beta_{n,2}\nabla f(T_2U_{n,3}x)+(1-\beta_{n,2})\nabla f(x)]\big),\\
W_n=U_{n,1}x&=\nabla f^*[\beta_{n,1}\nabla f(T_1U_{n,2}x)+(1-\beta_{n,1})\nabla f(x)],
\end{align*}
for all $x\in C$, where $proj_C^f$ is the Bregman projection from $E$ onto $C$. Such a mapping $W_n$ is called the Bergman $W$-mapping generated by $T_n,T_{n-1}, ..., T_1$ and $\beta_{n,n}, \beta_{n,n-1}, ..., \beta_{n,1}$.
\end{de}
\begin{pr}\cite{p19a}\label{3efrg}
Let $E$ be a reflexive Banach space and $f:E\rightarrow \mathbb{R}$ be a convex, continuous, strongly coercive and $G\hat{a}teaux$ differentiable function which is bounded on bounded subsets and uniformly convex on bounded subsets of $E$. Let $C$ be a nonempty, closed and convex subset of $E$. Let $T_1, T_2, ..., T_n$ be Bregman weak relatively nonexpansive mappings of $C$ into itself such that $\bigcap_{i=1}^nF(T_i)\neq \emptyset$, and let $\{\beta_{n,k}:k,n\in \mathbb{N}, 1\leq k\leq n\}$ be a sequence of real numbers such that $0<\beta_{n,1}\leq 1$ and $0<\beta_{n,i}<1$ for every $i=2, 3, ..., n$. Let $W_n$ is the Bregman $W$-mapping generated by $T_n, T_{n-1}, ..., T_1$ and  $\beta_{n,n}, \beta_{n,n-1}, ..., \beta_{n,1}$. Then the following assertions hold:
\begin{itemize}
  \item [{\rm(i)}] $F(W_n)=\bigcap_{i=1}^nF(T_i)$,
  \item [{\rm(ii)}] for every $k=1, 2, ..., n, \; x\in C$ and $z\in F(W_n),D_f(z,U_{n,k}x)\leq D_f(z,x)$ and\\ $D_f(z,T_kU_{n,k+1}x)\leq D_f(z,x)$,
  \item [{\rm(iii)}] for every $n\in \mathbb{N} $, $W_n$ is a Bregman weak relatively nonexpansive mapping.
  \end{itemize}
\end{pr}
\begin{lem}\cite{p14.52}\label{fg9jd}
  Let C be a nonempty closed convex subset of a reflexive Banach space E, $A:C\rightarrow E^*$ be a mapping, and $f:E\rightarrow \mathbb{R}$ be a Legendre function. Then
  \begin{equation*}
    Proj^f_C(\nabla f^*[\nabla f(x)-\alpha A(y)])=argmin_{\omega\in C} \{\alpha\langle\omega-y,A(y)\rangle+D_f(\omega,x)\}.
  \end{equation*}
  for all $x\in E$, $y\in C$ and $\lambda\in (0,\infty)$.
\end{lem}
\section{\textbf{Main results}}
 \begin{thm}\label{asli}
Let $E$ be a real reflexive Banach space. Suppose $f:E\rightarrow \mathbb{R}$ is a proper, convex, strongly coercive, Legendre function which is bounded on bounded subsets of $E$, uniformly Fr\'{e}chet differentiable, totally convex on bounded subsets of $E$. Let $C$ be a nonempty, closed and convex subset of $int\,domf$. Let for $j=1, 2, \ldots, N$, $h_j:C\times C\rightarrow \mathbb{R}$ be a bifunction satisfying $A1- A5$. Suppose $\{T_n\}_{n\in \mathbb{N}}$ is a family of Bregman weak relatively nonexpansive mappings of $C$ into itself and let $\{\beta_{n,k}:k,n\in \mathbb{N}, 1\leq k\leq n\}$ be a sequence of real numbers such that $0<\beta_{n,1}\leq 1$ and $0<\beta_{n,i}<1$ for all $n\in \mathbb{N}$ and $i=2, 3, ..., n$. Suppose $W_n$ be the Bregman $W$-mapping generated by $T_n, T_{n-1}, ..., T_1$ and $\beta_{n,n}, \beta_{n,n-1}, ..., \beta_{n,1}$. Let $S:C\rightarrow C$ be a Bregman quasi-nonexpansive mapping and demiclosed mapping. Suppose $\{B_i\}_{i=1}^2$ is a family of two Bregman inverse strongly monotone mappings of $C$ into $E$ and $\{B_{i,{\eta_n}}^f\}_{i=1}^2$ is the family of anti resolvent mappings of $\{B_i\}_{i=1}^2$. Let $G:E\rightarrow 2^{E^*}$ be a maximal monotone mapping on $E$ and let $ Q_\eta=Res_{\eta G}^f = (\nabla f+\eta G)^{-1}\nabla f$ be the resolvent of $G$ for $\eta>0$. Assume that
 \begin{equation*}
  \Omega =\big(\cap_{j=1}^N EP(h_j)\big)\bigcap F(S)\bigcap \big(\cap_{r=1}^\infty F(T_r)\big)\bigcap \big(\cap_{i=1}^2(B_i+G)^{-1}0^*\big)\neq \emptyset.
  \end{equation*}
Let
$\{x_n\}$, $\{y_n\}$, $\{z_n\}$ and $\{\bar{v}_n\}$ be the sequences defined by
   \begin{align}\label{mnzx}
  \begin{cases}
  x_1=x\in C, \quad chosen\;\; arbitrarily,\\
  C_1=C,\\
  t^j_n= argmin \big\{\alpha_nh_j(x_n,u)+D_f(u, x_n) : u\in C\big\},\quad j=1,2,\ldots, N,\\
  v^j_n= argmin \big\{\alpha_nh_j(t^j_n,u)+D_f(u, x_n): u\in C\big\},\quad j=1,2,\ldots, N,\\
  j_n\in Argmax\big\{D_f(v^j_n, x_n), j=1,2,\ldots, N\big\},\quad \bar{v}_n=v^{j_n}_n,\\
  y_n=\nabla f^{*}\big((1-\lambda_n)\nabla fW_n(\bar{v}_n)+\lambda_n\nabla f S(\bar{v}_n)\big),\\
  z_n=\nabla f^{*}\big((1-\sigma)\nabla f Q_{\eta_n}B_{1,{\eta_n}}^f(y_n)+\sigma\nabla f Q_{\eta_n}B_{2,{\eta_n}}^f(y_n)\big),\\
  C_n=\big\lbrace z\in C: D_f(z,\bar{v}_n)\leq D_f(z,x_n)\big\rbrace,\\
  D_n=\big\lbrace z\in E: D_f(z,y_n)\leq D_f(z,\bar{v}_n)\big\rbrace,\\
  M_n=\big\lbrace z\in E: D_f(z,z_n)\leq D_f(z,y_n)\big\rbrace,\\
  x_{n+1}=Proj_{C_n\cap D_n\cap M_n}^f(x_1),\quad \forall n\in \mathbb{N},
  \end{cases}
  \end{align}
where $\{\eta_n\}\subseteq(0,+\infty)$, $0<\sigma<1$ and $a,b\in \mathbb{R}$ be such that for all $n\in \mathbb{N}$, $0<a\leq\lambda_n\leq b<1$ and also, $\{\alpha_n\}\subseteq [r, s]\subseteq (0, p)$, where $p=min\{\frac{1}{c_1}, \frac{1}{c_2}\}$, $c_1=max_{1\leq j\leq N}c_{j,1}$, $c_2=max_{1\leq j\leq N}c_{j,2}$ where $c_{j,1}, c_{j,2}$ are the Bregman-Lipschitz coeffients of $h_j$ for all $1\leq j\leq N$.\\
Then $\{x_n\}$ converges strongly to a point $\omega_0\in \Omega$ where $\omega_0=Proj_\Omega^f(x_1)$.
  \end{thm}
  \begin{proof}
   We divide the proof into several steps:

   $Step\,1$: First, we prove that $\Omega$ is a closed and convex subset of $C$.\\
    Since $S$ is a Bregman quasi-nonexpansive mapping, by Lemma \ref{qctas} and the condition $\Omega\neq\emptyset$, $F(S)$ is nonempty, closed and convex. Also, for $i=1,2$ , it follows fom (i)-(ii) of Lemma \ref{vtye}  that $(B_i+G)^{-1}0^*=F(Q_{\eta_n}B_{i,\eta_n}^f)$ and $Q_{\eta_n}B_{i,\eta_n}^f$ is a Bregman strongly nonexpansive mapping and therefore we have that
    \begin{equation*}
      F(Q_{\eta_n}B_{i,\eta_n}^f)=\hat{F}(Q_{\eta_n}B_{i,\eta_n}^f).
    \end{equation*}
    Thus $\{Q_{\eta_n} B_{i,\eta_n}^f\}_{i=1}^2$ is a family of Bregman quasi-nonexpansive mappings. Using $\Omega\neq \emptyset$ and Lemma \ref{qctas}, we have that $(B_i+G)^{-1}0^*=F(Q_{\eta_n} B_{i,\eta_n}^f)$ is a nonempty, closed and convex subset of $C$. We know from Lemma \ref{qctas} that $\cap_{j=1}^\infty F(T_j)$ is closed and convex and also we have from Lemma \ref{aspoxc} that $\cap_{i=1}^N EP(h_i)$ is closed and convex. Then, $\Omega$ is nonempty, closed and convex. Therefore $Proj_\Omega^f$ is well defined.

    $Step\,2$: We prove that $C_n, D_n$ and $M_n$ are closed and convex subsets of $C$ and $\Omega\subseteq C_n\cap D_n\cap M_n$, for each $n\in \mathbb{N}$.\\
    In fact, it is clear that $C_1=C$ and then $C_1$ is closed and convex. We conclude that
    \begin{align*}
    D_f(z,\bar{v}_n)\leq D_f(z,x_n)\Leftrightarrow\langle\nabla f(x_n),z-x_n\rangle-\langle\nabla f(\bar{v}_n),z-\bar{v}_n\rangle\leq f(\bar{v}_n)-f(x_n),
    \end{align*}
     for all $z\in C$. Similarly, we have that
    \begin{align*}
    D_f(z,y_n)&\leq D_f(z,\bar{v}_n)\\
    &\Leftrightarrow f(z)-f(y_n)-\langle\nabla f(y_n),z-y_n\rangle\leq f(z)-f(\bar{v}_n)-\langle\nabla f(\bar{v}_n), z-\bar{v}_n\rangle\\
    &\Leftrightarrow\langle\nabla f(\bar{v}_n),z-\bar{v}_n\rangle-\langle\nabla f(y_n),z-y_n\rangle\leq f(y_n)-f(\bar{v}_n),
    \end{align*}
     for all $z\in E$. Also
        \begin{align*}
    D_f(z,z_n)\leq D_f(z,y_n)\Leftrightarrow\langle\nabla f(y_n),z-(y_n)\rangle-\langle\nabla f(z_n),z-z_n\rangle\leq f(z_n)-f(y_n),
    \end{align*}
     for all $z\in E$. Thus from the fact that $D_f(.,x)$ is continuous for each fixed $x$  and using the above inequalities $\{z\in C: D_f(z,\bar{v}_n)\leq D_f(z,x_n)\}$ and $\{z\in E: D_f(z,y_n)\leq D_f(z,\bar{v}_n)\}$ and $\{z\in E: D_f(z,z_n)\leq D_f(z,y_n)\}$ are closed and convex. Therefore, $C_n$ is a closed and convex subset of $C$ and also $D_n$ and $M_n$ are closed and convex subsets of $E$. Thus, $C_n\cap D_n\cap M_n$ is a closed and convex subset of $C$ for all $n\in \mathbb{N}$.

    Next, we show that $\Omega \subseteq C_n\cap D_n\cap M_n$ for all $n\geq 1$. Clearly $\Omega \subseteq C_1=C$. In fact, from Lemma \ref{swdczg} and our conditions, we have that
    \begin{align}\label{rkmcn}
     D_f(z,\bar{v}_n)\leq D_f(z,x_n)-(1-\alpha_n c_1)D_f(t^j_n,x_n)-(1-\alpha_n c_2)D_f(\bar{v}_n,t^j_n)\leq D_f(z,x_n),
    \end{align}
     for all $z\in \cap_{j=1}^N EP(h_j)$. Since $\cap_{j=1}^N EP(h_j)\subseteq C_n$, then $\Omega \subseteq C_n$ for all $n\in \mathbb{N}$.  From Lemma \ref{awqqq}, we have that
  \begin{align}\label{mner}
    D_f(z,y_n)&= D_f\bigg(z,\nabla f^{*}\big((1-\lambda_n)\nabla fW_n(\bar{v}_n)+\lambda_n\nabla f S(\bar{v}_n)\big)\bigg)\nonumber\\
    &\leq(1-\lambda_n)D_f\big(z,W_n(\bar{v}_n)\big)+\lambda_nD_f\big(z,S(\bar{v}_n)\big)\\
    &\leq(1-\lambda_n)D_f(z,\bar{v}_n)+\lambda_nD_f(z,\bar{v}_n)=D_f(z,\bar{v}_n),\nonumber
  \end{align}
  for all $z\in F(S)\bigcap \big(\cap_{r=1}^\infty F(T_r)\big)$ and for all $n\in \mathbb{N}$. Therefore $\Omega \subseteq D_n$ for all $n\in \mathbb{N}$. Furthermore, since $B_i$ is a Bregman inverse strongly monotone mapping for $i=1,2$, from Lemmas \ref{vtye} and \ref{awqqq} we have that
   \begin{align}\label{mnyu}
    D_f(z,z_n)&=D_f\bigg(z,\nabla f^{*}\big((1-\sigma)\nabla f Q_{\eta_n}B_{1,{\eta_n}}^f(y_n)+\sigma\nabla f Q_{\eta_n}B_{2,{\eta_n}}^f(y_n)\big)\bigg)\nonumber\\
    &\leq (1-\sigma)D_f\big(z,Q_{\eta_n}B_{1,{\eta_n}}^f(y_n)\big)+\sigma D_f\big(z,Q_{\eta_n}B_{2,{\eta_n}}^f(y_n)\big)\\
    &\leq (1-\sigma)D_f(z,y_n)+\sigma D_f(z,y_n)= D_f(z,y_n),\nonumber
  \end{align}
  for all $z\in \cap_{i=1}^2(B_i+G)^{-1}0^*$ and for all $n\in \mathbb{N}$. Therefore, $\Omega \subseteq M_n$ for all $n\in \mathbb{N}$. Thus, $\Omega \subseteq C_n\cap D_n\cap M_n$ for all $n\geq 1$.

  $Step\,3$: We show that $\displaystyle\lim_{n\rightarrow \infty} D_f(x_n,x_1)$ and $\displaystyle\lim_{n\rightarrow \infty}x_n$  exist.\\
  Since $\Omega$ is nonempty, closed and convex, there exists an element $\omega_0\in \Omega$ such that $\omega_0=Proj_{\Omega}^f(x_1)$. From $x_{n+1}=Proj_{C_n\cap D_n\cap M_n}^f(x_1)$, we get that
  \begin{equation*}
    D_f(x_{n+1},x_1)\leq D_f(z,x_1),
  \end{equation*}
  for all $z\in C_n\cap D_n\cap M_n$. Then from $\omega_0\in \Omega \subseteq C_n\cap D_n\cap M_n$, we obtain that
  \begin{equation}\label{bher}
  D_f(x_{n+1},x_1)\leq D_f(\omega_0,x_1).
  \end{equation}
  This shows that $\{D_f(x_n,x_1)\}$ is a bounded sequence. By (iii) of Lemma \ref{rtmnw}, we have that
  \begin{align}\label{jkwe}
  0&\leq D_f(x_{n+1},x_n)=D_f\big(x_{n+1},Proj_{C_n\cap D_n\cap M_n}^f(x_1)\big)\nonumber\\
  &\leq D_f(x_{n+1},x_1)-D_f\big(Proj_{C_n\cap D_n\cap M_n}^f(x_1),x_1\big)\\
  &=D_f(x_{n+1},x_1)-D_f(x_n,x_1),\nonumber
  \end{align}
  for all $n\geq 1$. Therefore
  \begin{equation}\label{fgrt}
  D_f(x_n,x_1)\leq D_f(x_{n+1},x_1).
  \end{equation}
  This implies that $\{D_f(x_n,x_1)\}$ is bounded and nondecreasing. Then  $\displaystyle\lim_{n\rightarrow \infty}D_f(x_n,x_1)$ exists. In view of Lemma \ref{qwpo}, we deduce that the sequence $\{x_n\}$ is bounded. Also, from \eqref{jkwe}, we have that
  \begin{equation}\label{dfpo}
  \lim_{n\rightarrow \infty} D_f(x_{n+1},x_n)=0.
  \end{equation}
  Since the function $f$ is totally convex on bounded sets, by Lemma \ref{qpocv}, we have that
  \begin{equation}\label{asol}
  \lim_{n\rightarrow \infty}\|x_{n+1}-x_n\|=0.
  \end{equation}
   Thus $\{x_n\}$ is a Cauchy sequence. Then, we conclude that there exists $\bar{x}\in C$ such that
  \begin{align}\label{awlo}
   \displaystyle\lim_{n\rightarrow \infty}\|x_n-\bar{x}\|=0.
  \end{align}

  $Step\,4$: We prove that $\displaystyle\lim_{n\rightarrow \infty}\|\nabla f(x_n)-\nabla f(v_n^j)\|=0$, for all $1\leq j\leq N$.\\
  From $x_{n+1}\in C_n$, we get that
  \begin{align}\label{azlk}
  D_f(x_{n+1},\bar{v}_n)\leq D_f(x_{n+1},x_n),
  \end{align}
  and using \eqref{dfpo} we have
  \begin{align}\label{3qjhdv}
   \displaystyle\lim_{n\rightarrow \infty}D_f(x_{n+1},\bar{v}_n)=0.
  \end{align}
   We conclude from \eqref{awlo}, the condition A3 and Lemma \ref{zakhv} that $\{\bar{v}_n\}_{n\in \mathbb{N}}$ is bounded. From the totally convexity of $f$ on bounded sets, by Lemma \ref{qpocv} and \eqref{3qjhdv}, we have that
   \begin{equation}\label{3afgsqv}
  \lim_{n\rightarrow \infty}\|x_{n+1}-\bar{v}_n\|=0.
  \end{equation}
    Using \eqref{asol} and \eqref{3afgsqv}, we conclude that
  \begin{align}\label{astp}
  \lim_{n\rightarrow \infty}\|x_n-\bar{v}_n\|=0.
  \end{align}
Since the function $f$ is continuous on $E$, from \eqref{astp},  $\displaystyle\lim_{n\rightarrow \infty}(f(\bar{v}_n)-f(x_n))=0$, and so, from the definition of the Bregman distance, we obtain $\displaystyle\lim_{n\rightarrow \infty} D_f(\bar{v}_n,x_n)=0$. Using the definition of $j_n$, we conclude that $\displaystyle\lim_{n\rightarrow \infty} D_f(v_n^j,x_n)=0$, and therefore from Lemma \ref{qpocv}, $\displaystyle\lim_{n\rightarrow \infty}\|v_n^j-x_n\|=0$. From Lemma \ref{zakhv}, $\nabla f$ is uniformly continuous on bounded subsets of $E$, then we have that
  \begin{equation}\label{awmx}
  \lim_{n\rightarrow \infty}\|\nabla f(v_n^j)-\nabla f(x_n)\|=0,
 \end{equation}
 for all $1\leq j\leq N$.

  $Step\,5$: We show that $\displaystyle\lim_{n\rightarrow \infty}\|\nabla f(\bar{v}_n)-\nabla f(y_n)\|=0$.\\
  From $x_{n+1}\in C_n\cap D_n$, we get that
  \begin{align}\label{azlk}
  D_f(x_{n+1},y_n)\leq D_f(x_{n+1},x_n),
  \end{align}
  and using \eqref{dfpo} we have
  \begin{align}\label{3wedv}
   \displaystyle\lim_{n\rightarrow \infty}D_f(x_{n+1},y_n)=0.
  \end{align}
     The function $f$ is bounded on bounded subsets of $E$ and so $\nabla f$ and $\nabla f^*$ are also bounded on bounded subsets of $E^*$ and $E$ respectively (see, for example, \cite[Proposition 1.1.11]{p11}, for more details). The boundedness of $\{x_n\}_{n\in \mathbb{N}}$ implies that $\{D_f(z,x_n)\}_{n\in \mathbb{N}}$ is bounded for all $z\in E$ and from the definitions of Bregman weak relatively nonexpansive mappings and Bregman quasi nonexpansive mappings, $\{D_f(p,W_n(\bar{v}_n))\}_{n\in \mathbb{N}}$ and $\{D_f(p,S(\bar{v}_n))\}_{n\in \mathbb{N}}$ are bounded  for all $p\in F(S)\cap (\bigcap_{r=1}^\infty F(T_r))$, respectively and therefore from Lemma \ref{fdyuq}, $\{S(\bar{v}_n)\}_{n\in \mathbb{N}}$ and $\{W_n(\bar{v}_n)\}_{n\in \mathbb{N}}$ are bounded. This implies that the sequences $\{\nabla fS(\bar{v}_n)\}_{n\in \mathbb{N}}$ and $\{\nabla fW_n(\bar{v}_n)\}_{n\in \mathbb{N}}$ are bounded in $E^*$. Therefore, $\{y_n\}$ is bounded. From the totally convexity of $f$ on bounded sets, by Lemma \ref{qpocv} and \eqref{3wedv}, we have that
  \begin{equation*}
  \lim_{n\rightarrow \infty}\|x_{n+1}-y_n\|=0,
  \end{equation*}
  then using \eqref{asol}, we get that
  \begin{align}\label{astp1g}
  \lim_{n\rightarrow \infty}\|x_n-y_n\|=0,
  \end{align}
  hence, applying \eqref{astp}, we conclude that
    \begin{align}\label{abngher}
  \lim_{n\rightarrow \infty}\|\bar{v}_n-y_n\|=0.
  \end{align}
   From Lemma \ref{zakhv}, $\nabla f$ is uniformly continuous on bounded subsets of $E$,  then we have from \eqref{abngher} that
  \begin{equation}\label{ertawmx}
  \lim_{n\rightarrow \infty}\|\nabla f(\bar{v}_n)-\nabla f(y_n)\|=0.
  \end{equation}

 $Step\,6$: We prove that $\displaystyle\lim_{n\rightarrow \infty}\|\nabla f(y_n)-\nabla f(z_n)\|=0$.\\
  Using \eqref{3wedv} and from $x_{n+1}\in M_n$, we get that
  \begin{align}\label{aedm}
  D_f(x_{n+1},z_n)\leq D_f(x_{n+1},y_n)\rightarrow 0\,\;(as\;\;n\rightarrow \infty).
  \end{align}
   The boundedness of $\{y_n\}_{n\in \mathbb{N}}$ implies that $\{D_f(z,y_n)\}_{n\in \mathbb{N}}$  is bounded for all $z\in E$. Hence, from the definition of Bregman strongly nonexpansive mapping and Lemma \ref{vtye}, $\{D_f(p,Q_{\eta_n}B_{i,{\eta_n}}^f(y_n))\}_{n\in \mathbb{N}}$ is bounded for $i=1, 2$ and for all $p\in \cap_{i=1}^2(B_i+G)^{-1}0^*$. Therefore from Lemma \ref{fdyuq}, $\{Q_{\eta_n}B_{i,{\eta_n}}^f(y_n)\}_{n\in \mathbb{N}}$ is bounded for $i=1, 2$. This implies that the sequences $\{\nabla fQ_{\eta_n}B_{i,{\eta_n}}^f(y_n)\}_{n\in \mathbb{N}}$ is bounded for $i=1, 2$. Therefore, $\{z_n\}$ is bounded.
From the totally convexity of $f$ on bounded sets, by Lemma \ref{qpocv} and \eqref{aedm}, we conclude that
  \begin{align}\label{aqhn}
  \lim_{n\rightarrow \infty}\|x_{n+1}-z_n\|=0.
  \end{align}
   Using \eqref{asol} and \eqref{aqhn}, we have that
  \begin{align}\label{aqhnui}
  \lim_{n\rightarrow \infty}\|x_n-z_n\|=0.
  \end{align}
  Applying \eqref{astp1g} and \eqref{aqhnui}, we get that
  \begin{align}\label{aqekgi}
  \lim_{n\rightarrow \infty}\|y_n-z_n\|=0.
  \end{align}
  Then from the uniformly continuity of $\nabla f$, we have that
  \begin{align}\label{awlfm}
  \lim_{n\rightarrow \infty}\|\nabla f(y_n)-\nabla f(z_n)\|=0.
  \end{align}

  $Step\,7$: We prove that $\bar{x}\in \Omega$.\\
  First, we show that $\bar{x}\in \cap_{j=1}^N EP(h_j)$. Using Lemma \ref{swdczg}, we conclude that
  \begin{equation}\label{bnhej}
    D_f(\omega_0,v^j_n)\leq D_f(\omega_0,x_n)-(1-\alpha_n c_1)D_f(t^j_n,x_n)-(1-\alpha_n c_2)D_f(v^j_n,t^j_n).
  \end{equation}
  It follows from \eqref{awpon}, \eqref{bnhej} and our conditions that
  \begin{align}\label{spongh}
    (1-\alpha_n c_1)D_f(t^j_n,x_n)&\leq D_f(\omega_0,x_n)-D_f(\omega_0,v^j_n)\nonumber\\
   &\leq D_f(\omega_0,x_n)-D_f(\omega_0,v^j_n)+D_f(x_n,v^j_n)\nonumber\\
   &=\langle \omega_0-x_n, \nabla f(v^j_n)-\nabla f(x_n)\rangle.
  \end{align}
    From \eqref{awmx}, \eqref{spongh} and the boundedness of the sequence $\{x_n\}_{n\in \mathbb{N}}$ and our conditions, we get $\lim_{n\rightarrow \infty}D_f(t^j_n,x_n)=0$, and therefore by Lemma \ref{qpocv},
    \begin{align}\label{asfbtm}
  \lim_{n\rightarrow \infty}\|t^j_n-x_n\|=0.
  \end{align}
   for all $1\leq j\leq N$.
  Also, using the fact that $D_f(., x)$ is convex, differentiable and lower semicontinuous on $\textit{int dom}f$, it follows from Proposition \ref{aplenfg} that $\partial_1D_f(x, y)=\{\nabla f(x)-\nabla f(y)\}$. Since
  \begin{equation*}
  t^j_n= argmin \big\{\alpha_nh_j(x_n,u)+D_f(u, x_n): u\in C\big\}.
  \end{equation*}
  by Lemmas \ref{khecpo}, \ref{ehfbqw} and the condition A4, we obtain
  \begin{align*}
   0\in \partial \{\alpha_n h_j(x_n, t^j_n)&+D_f(t^j_n,x_n)\}+N_C(t^j_n)\\
   &=\partial_2\alpha_n h_j(x_n, t^j_n)+\partial_1D_f(t^j_n,x_n)+N_C(t^j_n).
  \end{align*}
  Therefore, there exist $\xi^j_n \in \partial_2 h_j(x_n, t^j_n)$ and $\bar{\xi}^j_n \in N_C(t^j_n)$, such that
  \begin{equation}\label{kjhhbc}
   \alpha_n\xi^j_n+ \nabla f(t^j_n)- \nabla f(x_n)+ \bar{\xi}^j_n=0.
  \end{equation}
  Since $\bar{\xi}^j_n\in N_C(t^j_n)$, then $\langle y-t^j_n,\bar{\xi}^j_n\rangle \leq0$ for all $y\in C$. This together with \eqref{kjhhbc} implies that
  \begin{equation}\label{ekjng}
  \alpha_n\langle y-t^j_n,\xi^j_n\rangle \geq \langle t^j_n-y, \nabla f(t^j_n)-\nabla f(x_n)\rangle,\quad \forall y\in C.
  \end{equation}
  Since $\xi^j_n \in \partial_2 h_j(x_n, t^j_n)$, we have
  \begin{equation}\label{kjxcer}
   h_j(x_n,y)-h_j(x_n,t^j_n)\geq \langle y-t^j_n,\xi^j_n\rangle,\quad \forall y\in C.
  \end{equation}
  Using \eqref{ekjng} and \eqref{kjxcer}, we get
  \begin{equation}\label{dkvbqw}
    \alpha_n[h_j(x_n,y)-h_j(x_n,t^j_n)]\geq \langle t^j_n-y, \nabla f(t^j_n)-\nabla f(x_n)\rangle,\quad \forall y\in C.
  \end{equation}
  From \eqref{awlo} and \eqref{asfbtm}, we find $t^j_n\rightarrow \bar{x}$. Letting $n\rightarrow \infty$ in
\eqref{dkvbqw} and using A3 and our conditions, we conclude that $h_j(\bar{x},y)\geq 0$ for all $y\in C$ and any $j = 1, 2, \ldots, N$. Thus, $\bar{x}\in \cap_{j=1}^N EP(h_j)$.

Since $S$ is a Bregman quasi-nonexpansive mapping and $W_n$ is a $W$-mapping for all $n\in \mathbb{N}$ and $0<a\leq\lambda_n\leq b<1$, then we get that
   \begin{align*}
  \langle \bar{v}_n-z, \nabla f(\bar{v}_n)-\nabla f(y_n)\rangle
  &=\big\langle \bar{v}_n-z,\nabla f(\bar{v}_n)-\nabla f\nabla f^*\big((1-\lambda_n)\nabla fW_n(\bar{v}_n)+\lambda_n\nabla f S(\bar{v}_n)\big)\big\rangle\\
  &=\big\langle \bar{v}_n-z,\nabla f(\bar{v}_n)-\big((1-\lambda_n)\nabla fW_n(\bar{v}_n)+\lambda_n\nabla f S(\bar{v}_n)\big)\big\rangle\\
  &=(1-\lambda_n)\big\langle \bar{v}_n-z,\nabla f(\bar{v}_n)-\nabla fW_n(\bar{v}_n)\big\rangle\\
  &\qquad\qquad\qquad\qquad\;\;+\lambda_n\big\langle \bar{v}_n-z,\nabla f(\bar{v}_n)-\nabla f S(\bar{v}_n)\big\rangle\\
  &= (1-\lambda_n)\big(D_f(z,\bar{v}_n)+D_f(\bar{v}_n,W_n(\bar{v}_n))-D_f(z,W_n(\bar{v}_n))\big)\\
  &\qquad\qquad\qquad\qquad\;\;+\lambda_n \big(D_f(z,\bar{v}_n)+D_f(\bar{v}_n,S(\bar{v}_n))-D_f(z,S(\bar{v}_n))\big)\\
   &\geq (1-b)\big(D_f(z,\bar{v}_n)+D_f(\bar{v}_n,W_n(\bar{v}_n))-D_f(z,\bar{v}_n)\big)\\
  &\qquad\qquad\qquad\qquad\;\;+a \big(D_f(z,\bar{v}_n)+D_f(\bar{v}_n,S(\bar{v}_n))-D_f(z,\bar{v}_n)\big)\\
  &=(1-b)D_f(\bar{v}_n,W_n(\bar{v}_n))+aD_f(\bar{v}_n,S(\bar{v}_n)),
  \end{align*}
 for all $z\in \bigcap_{i=1}^{\infty}F(T_{i}) \cap F(S)$. Then, we have from \eqref{ertawmx} and continuity $D_f$  that
   $ \displaystyle\lim_{n\rightarrow\infty}D_f(\bar{v}_n,S(\bar{v}_n))=0$ and $\displaystyle\lim_{n\rightarrow\infty}D_f(\bar{v}_n,W_n(\bar{v}_n))=0$. Then, we have from Lemma \ref{qpocv} that
   \begin{align}\label{aqoxg}
  \lim_{n\rightarrow \infty}\|\bar{v}_n-S(\bar{v}_n)\|=0,
  \end{align}
  and also
  \begin{align}\label{aqoxg2}
  \lim_{n\rightarrow \infty}\|\bar{v}_n-W_n(\bar{v}_n)\|=0.
  \end{align}
   On the other hand, from \eqref{awlo} and \eqref{astp}, we have that
   \begin{align}\label{3xhkn}
    \lim_{n\rightarrow \infty}\|\bar{v}_n-\bar{x}\|=0.
   \end{align}
   Since $S$ is demiclosed and from \eqref{aqoxg} and \eqref{3xhkn}, we conclude that $\bar{x}\in F(S)$. Also from \eqref{aqoxg2} and \eqref{3xhkn}, we conclude $\bar{x}\in \tilde{F}(W_n)$ and from Proposition \ref{3efrg}, $W_n$ is  Bregman weak relatively nonexpansive, hence, we get that $\bar{x}\in\bigcap_{i=1}^nF(T_{i})$ for all $n\in \mathbb{N}$. Therefore $\bar{x}\in\bigcap_{i=1}^{\infty}F(T_{i})$.
   Now, we show that $\bar{x}\in\cap_{i=1}^2(B_i+G)^{-1}0^*$. From \eqref{awpon} and (iii) of Lemma \ref{vtye}, we get that
    \begin{align*}
 \langle y_n-z,\nabla f(y_n) &-\nabla f(z_n)\rangle\\
  &=\big\langle y_n-z,\nabla f(y_n)-\nabla f\nabla f^{*}\big((1-\sigma)\nabla f Q_{\eta_n}B_{1,{\eta_n}}^f(y_n)+\sigma\nabla f Q_{\eta_n}B_{2,{\eta_n}}^f(y_n)\big)\rangle\\
   &=(1-\sigma)\big\langle y_n-z,\nabla f(y_n)-\nabla f Q_{\eta_n}B_{1,{\eta_n}}^f(y_n)\big\rangle\\
   &\qquad\qquad\qquad\,+\sigma\big\langle y_n-z,\nabla f(y_n)-\nabla f Q_{\eta_n}B_{2,{\eta_n}}^f(y_n)\big\rangle\\
  &=(1-\sigma)\big(D_f(z,y_n)+D_f(y_n,Q_{\eta_n}B_{1,{\eta_n}}^fy_n)-D_f(z,Q_{\eta_n}B_{1,{\eta_n}}^fy_n)\big)\\
  &\qquad\qquad\qquad\,+\sigma\big(D_f(z,y_n)+D_f(y_n,Q_{\eta_n}B_{2,{\eta_n}}^fy_n)-D_f(z,Q_{\eta_n}B_{2,{\eta_n}}^fy_n)\big)\\
  &\geq (1-\sigma)\big(D_f(z,y_n)+D_f(y_n,Q_{\eta_n}B_{1,{\eta_n}}^fy_n)-D_f(z,y_n)\big)\\
  &\qquad\qquad\qquad\,+\sigma\big(D_f(z,y_n)+D_f(y_n,Q_{\eta_n}B_{2,{\eta_n}}^fy_n)-D_f(z,y_n)\big)\\
 &=(1-\sigma)D_f(y_n,Q_{\eta_n}B_{1,{\eta_n}}^fy_n)+\sigma D_f(y_n,Q_{\eta_n}B_{2,{\eta_n}}^fy_n)
  \end{align*}
  for all $z\in \cap_{i=1}^2(B_i+G)^{-1}0^*$. Then, we have from \eqref{awlfm} that
  $\displaystyle\lim_{n\rightarrow \infty}D_f(y_n,Q_{\eta_n}B_{i,{\eta_n}}^fy_n)=0$ for $i=1,2$. Then, we have from Lemma \ref{qpocv}, that
  \begin{align}\label{azjku}
  \displaystyle\lim_{n\rightarrow \infty}\|y_n-Q_{\eta_n}B_{i,{\eta_n}}^fy_n\|=0,
  \end{align}
  for $i=1,2$. On the other hand, it follows from \eqref{awlo} and \eqref{astp1g} that
  \begin{align}\label{aoeng}
  \displaystyle\lim_{n\rightarrow \infty}\|y_n-\bar{x}\|=0.
  \end{align}
  Now, from \eqref{azjku} and \eqref{aoeng}, we have that $\bar{x}\in\hat{F}(Q_{\eta_n}B_{i,{\eta_n}})$ for $i=1,2$. From Lemma \ref{vtye}, we conclude that
  \begin{align*}
  \hat{F}(Q_{\eta_n}B_{i,{\eta_n}})=F(Q_{\eta_n}B_{i,{\eta_n}})=(B_i+G)^{-1}0^*,\quad i=1,2,
  \end{align*}
  and therefore $\bar{x}\in\cap_{i=1}^2(B_i+G)^{-1}0^*$. Thus $\bar{x}\in\Omega$.
  Since $\omega_0=Proj_{\Omega}^f(x_1)$ and $\|x_n-\bar{x}\|\rightarrow0$, we have from \eqref{bher} that
   \begin{align*}
   D_f(\omega_0,x_1)\leq D_f(\bar{x},x_1)=\displaystyle\lim_{n\rightarrow \infty}D_f(x_n,x_1)\leq D_f(\omega_0,x_1),
   \end{align*}
   therefore,
  \begin{align}\label{qaswe}
   D_f(\bar{x},x_1)=D_f(\omega_0,x_1).
  \end{align}
   From \eqref{mnzx}, \eqref{awlo}, Lemma \ref{rtmnw} and $\omega_0\in \Omega\subseteq C_n \cap D_n \cap M_n$ for all $n\in\mathbb{N}$ and also from uniformly continuity on bounded subsets of $\nabla f$, we have that
   \begin{align*}\label{lkscv}
    \langle x_{n+1}-\omega_0,\nabla f(x_1)-\nabla f(x_{n+1})\rangle\geq 0 &\Rightarrow
    \lim_{n\rightarrow \infty}\langle x_{n+1}-\omega_0,\nabla f(x_1)-\nabla f(x_{n+1})\rangle\geq 0\\
    &\Rightarrow \langle \bar{x}-\omega_0,\nabla f(x_1)-\nabla f(\bar{x})\rangle\geq 0.
   \end{align*}
    Hence, from \eqref{awpon} and \eqref{qaswe}, we get that
    \begin{align*}
      D_f(\omega_0,\bar{x})&= D_f(\omega_0,\bar{x})+D_f(\bar{x},x_1)-D_f(\omega_0,x_1)=\langle \omega_0-\bar{x},\nabla f(x_1)-\nabla f(\bar{x})\rangle\leq 0.
    \end{align*}
    Therefore $D_f(\omega_0,\bar{x})=0$. Thus $\omega_0=\bar{x}$ and hence $x_n\rightarrow \omega_0$. This completes the proof.
  \end{proof}
    \section{\textbf{Applications}}
In this section, using Theorem \ref{asli}, a new strong convergence theorem in Banach spaces will be established. Let $E$ be a Banach space and $f(x)=\frac{1}{2}\|x\|^2$ for all $x\in E$, then we have that $\nabla f(x)=J(x)$ for all $x\in E$. The normalized duality mapping $J:E\rightarrow2^{E^*}$ is defined by
\begin{equation*}
J(x)=\{j(x)\in E^*:\langle j(x),x\rangle=\|x\|.\|j(x)\|\;,\;\|j(x)\|=\|x\|\}.
\end{equation*}
Hence, $D_f(.,.)$ reduces to the usual map $\frac{1}{2}\phi(.,.)$ as
\begin{equation}
D_f(x,y)=\frac{1}{2}\phi(x,y)=\frac{1}{2}\Vert x\Vert^{2}-\langle x,Jy\rangle +\frac{1}{2}\Vert y\Vert^{2}, \quad \forall x,y\in E.
\end{equation}
If $E$ is a Hilbert space, then $D_f(x,y)=\frac{1}{2}\|x-y\|^2$. For more details, we refer the reader to \cite{3qq11}.

Let $E$ be a smooth, strictly convex, and reflexive Banach space and $C$ a nonempty, closed and convex subset of $E$, then we have that $\Pi_C(x)=proj_C^{\frac{1}{2}\Vert x\Vert^{2}}(x)$ (see, e.g., \cite{3qq11}). Recall that the generalized projection $\Pi_C$ from $E$ onto $C$ is defined and denoted by
\begin{equation*}
  \Pi_C(x)=\text{arg}\,\min_{y\in C}\phi(y,x),\quad \forall x\in E.
\end{equation*}
We know that, if $E$ be a smooth, strictly convex, and reflexive Banach space and $C$ be a nonempty, closed and convex subset of $E$ and $x\in E$, $u\in C$, then we have that:
\begin{align*}
  u=\Pi_C(x)\Leftrightarrow \langle y-u,Jx-Ju\rangle \leq 0,\,\,\,\forall y\in C,
\end{align*}
(see \cite{3qq11,3qq112,p14.54}).

 It is known that $\partial f$ is a maximal monotone operator (see Rockafellar \cite{p23.1}). Let $C$ be a nonempty, closed and convex subset of $E$ and $i_C$ be the indicator function of $C$, i.e.,
\begin{equation*}
i_C(x)= \left\{
\begin{array}{r}
0 \quad x\in C,\\
\,\infty \quad x\notin C.
\end{array} \right.
\end{equation*}
 Then $i_C$ is a proper, lower semicontinuous and convex function on $E$ and hence, the subdifferential $ \partial i_C$ of $i_C$ is a maximal monotone operator. Therefore, the generalized resolvent $Q_\eta $ of $ \partial i_C$ for $\eta>0$ is defined as follows:
 \begin{equation}
 Q_\eta x=(J+\eta \partial i_C)^{-1}Jx, \quad \forall x\in E.
 \end{equation}
 For any $x\in E$ and $u\in C$, the following relations hold:
 \begin{align*}
 u=&Q_\eta(x) \Leftrightarrow Jx\in Ju+\eta \partial i_Cu\\
 &\Leftrightarrow \frac{1}{\eta}(Jx-Ju)\in  \partial i_Cu\\
 &\Leftrightarrow i_Cy\geq \langle y-u,\frac{1}{\eta}(Jx-Ju)\rangle +i_Cu,\,\,\,\forall y\in E\\
 &\Leftrightarrow 0\geq \langle y-u,\frac{1}{\eta}(Jx-Ju)\rangle ,\,\,\,\forall y\in C\\
 &\Leftrightarrow \langle y-u,Jx-Ju\rangle \leq 0,\,\,\,\forall y\in C\\
 &\Leftrightarrow u=\Pi_C(x).
  \end{align*}

  Let $f(x)=\frac{1}{2}\|x\|^2$ for all $x\in E$ and we replace $D_f(.,.)$ by $\frac{1}{2}\phi(.,.)$ and $\nabla f$ by $J$
   in Theorem \ref{asli}. Now the following theorem can be concluded from Theorem \ref{asli}:
  \begin{thm}\label{sdo3mn}
Let $E$ be a smooth, strictly convex and real reflexive Banach space. Let $C$ be a nonempty, closed and convex subset of $E$. Let for $j=1, 2, \ldots, N$, $h_j:C\times C\rightarrow \mathbb{R}$ be a bifunction satisfying $A1- A5$. Suppose $\{T_n\}_{n\in \mathbb{N}}$ is a family of weak relatively nonexpansive mappings of $C$ into itself and let $\{\beta_{n,k}:k,n\in \mathbb{N}, 1\leq k\leq n\}$ be a sequence of real numbers such that $0<\beta_{n,1}\leq 1$ and $0<\beta_{n,i}<1$ for all $n\in \mathbb{N}$ and $i=2, 3, ..., n$. Suppose $W_n$ is the $W$-mapping generated by $T_n, T_{n-1}, ..., T_1$ and $\beta_{n,n}, \beta_{n,n-1}, ..., \beta_{n,1}$. Let $S:C\rightarrow C$ be a quasi-nonexpansive mapping and demiclosed mapping. Suppose $\{B_i\}_{i=1}^2$ is a family of two inverse strongly monotone mappings of $C$ into $E$. Assume that

 \begin{equation*}
  \Omega =\big(\cap_{j=1}^N EP(h_j)\big)\bigcap F(S)\bigcap \big(\cap_{r=1}^\infty F(T_r)\big)\bigcap \big(\cap_{i=1}^2(B_i+\partial i_C)^{-1}0^*\big)\neq \emptyset.
  \end{equation*}
  Let $\{x_n\}$, $\{y_n\}$, $\{z_n\}$ and $\{\bar{v}_n\}$ be the sequences defined by
   \begin{align}\label{mnzxq}
  \begin{cases}
  x_1=x\in C, \quad chosen\;\; arbitrarily,\\
  C_1=C,\\
  t^j_n= argmin \big\{\alpha_nh_j(x_n,u)+\frac{1}{2}\phi(u, x_n) : u\in C\big\},\quad j=1,2,\ldots, N,\\
  v^j_n= argmin \big\{\alpha_nh_j(t^j_n,u)+\frac{1}{2}\phi(u, x_n): u\in C\big\},\quad j=1,2,\ldots, N,\\
  j_n\in Argmax\big\{\phi(v^j_n, x_n), j=1,2,\ldots, N\big\},\quad \bar{v}_n=v^{j_n}_n,\\
  y_n=J^{-1}\big((1-\lambda_n)JW_n(\bar{v}_n)+\lambda_nJS(\bar{v}_n)\big),\\
  z_n=J^{-1}\big((1-\sigma)J\Pi_CJ^{-1}(J-\eta_nB_1)(y_n)+\sigma J\Pi_CJ^{-1}(J-\eta_nB_2)(y_n)\big),\\
  C_n=\big\lbrace z\in C: D_f(z,\bar{v}_n)\leq D_f(z,x_n)\big\rbrace,\\
  D_n=\big\lbrace z\in E: D_f(z,y_n)\leq D_f(z,\bar{v}_n)\big\rbrace,\\
  M_n=\big\lbrace z\in E: D_f(z,z_n)\leq D_f(z,y_n)\big\rbrace,\\
  x_{n+1}=\Pi_{C_n\cap D_n\cap M_n}(x_1),\quad \forall n\in \mathbb{N},
  \end{cases}
  \end{align}
where $\{\eta_n\}\subseteq(0,+\infty)$, $0<\sigma<1$ and $a,b\in \mathbb{R}$ be such that for all $n\in \mathbb{N}$, $0<a\leq\lambda_n\leq b<1$ and also, $\{\alpha_n\}\subseteq [r, s]\subseteq (0, p)$, where $p=min\{\frac{1}{c_1}, \frac{1}{c_2}\}$, $c_1=max_{1\leq j\leq N}c_{j,1}$, $c_2=max_{1\leq j\leq N}c_{j,2}$ where $c_{j,1}, c_{j,2}$ are the Lipschitz coeffients of $h_j$ for all $1\leq j\leq N$.\\
Then $\{x_n\}$ converges strongly to a point $\omega_0\in \Omega$ where $\omega_0=\Pi_\Omega(x_1)$.
  \end{thm}
Next, we consider the particular equilibrium problem corresponding to the function $h$ defined for all $x,y\in C$ by $h(x,y)=\langle y-x,Ax\rangle$ with $A:C\rightarrow E^*$. Now, we consider the following classical variational inequality:
\begin{equation*}
  \text{Find}\quad z\in C\quad\text{such that}\quad \langle y-z,Az\rangle\geq0,\quad \forall y\in C.
\end{equation*}
The set of solutions of this problem is denoted by $VI(A,C)$. 
Let $E$ be a real Banach space and $1\leq q\leq 2\leq p$ with $\frac{1}{p}+\frac{1}{q}=1$ and Suppose $\delta_E$ is a modulus of convexity of E. The $E$ is called $p$-uniformly convex if there is a $c_p\geq0$ so that $\delta_E(\varepsilon)\geq c_p\varepsilon^p$ for all $\varepsilon\in (0,2]$. The modulus of smoothness $\rho_E(\tau):[0,\infty)\rightarrow[0,\infty)$ is defined by
\begin{equation*}
 \rho_E(\tau)=sup\big\{\frac{\|x+\tau y\|+\|x-\tau y\|}{2}-1:\;\|x\|=\|y\|=1\big\}.
\end{equation*}
where $E$ is called uniformly smooth if $\displaystyle\lim_{\tau\rightarrow0}\frac{\rho_E(\tau)}{\tau}=0$.
For the $p$-uniformly convex space, the metric and Bregman distance have the following relation \cite{p24q1a}:
\begin{equation}\label{kcv45t}
  \tau\|x-y\|^p\leq D_{\frac{1}{p}\|.\|^p}(x,y)\leq\langle x-y,J_p(x)-J_p(y)\rangle,
\end{equation}
where $\tau\geq0$ is a fixed number. We know that $E$ is smooth if and only if $J_p$ is single valued mapping of $E$ into $E^*$. We also know that $E$ is reflexive if and only if $J_p$ is surjective, and $E$ is strictly convex if and only if $J_p$ is one-to-one. Therefore, if $E$ is smooth, strictly convex and reflexive Banach space, then $J_p$ is a single-valued bijection and in this case, $J_p= J_{q^*}^{-1}$, where $J_{q^*}$ is the duality mapping of $E^*$.
 \begin{thm}
Let $E$ be a uniformly smooth, 2-uniformly convex and real reflexive Banach space. Suppose $C$ be a nonempty, closed and convex subset of $E$. Let that $\{A_j\}_{j=1}^N$ be a finite family of pseudomonotone and $L_i$-Lipschitz continuous mapping from $C$ to $E^*$.  Suppose $\{T_n\}_{n\in \mathbb{N}}$ is a family of weak relatively nonexpansive mappings of $C$ into itself and let $\{\beta_{n,k}:k,n\in \mathbb{N}, 1\leq k\leq n\}$ be a sequence of real numbers such that $0<\beta_{n,1}\leq 1$ and $0<\beta_{n,i}<1$ for all $n\in \mathbb{N}$ and $i=2, 3, ..., n$. Suppose $W_n$ is the $W$-mapping generated by $T_n, T_{n-1}, ..., T_1$ and $\beta_{n,n}, \beta_{n,n-1}, ..., \beta_{n,1}$. Let $S:C\rightarrow C$ be a quasi-nonexpansive mapping and demiclosed mapping. Suppose $\{B_i\}_{i=1}^2$ is a family of two inverse strongly monotone mappings of $C$ into $E$. Assume that

 \begin{equation*}
  \Omega =\big(\cap_{j=1}^N VI(A_j,C)\big)\bigcap F(S)\bigcap \big(\cap_{r=1}^\infty F(T_r)\big)\bigcap \big(\cap_{i=1}^2(B_i+\partial i_C)^{-1}0^*\big)\neq \emptyset.
  \end{equation*}
  Let $\{x_n\}$, $\{y_n\}$, $\{z_n\}$ and $\{\bar{v}_n\}$ be the sequences defined by
   \begin{align}\label{mnzkk}
  \begin{cases}
  x_1=x\in C, \quad chosen\;\; arbitrarily,\\
  C_1=C,\\
  t^j_n= \Pi_C[J^{-1}(J(x_n)-\alpha_nA_j(x_n))],\quad j=1,2,\ldots, N,\\
  v^j_n= \Pi_C[J^{-1}(J(x_n)-\alpha_nA_j(t^j_n))],\quad j=1,2,\ldots, N,\\
  j_n\in Argmax\big\{\phi(v^j_n, x_n), j=1,2,\ldots, N\big\},\quad \bar{v}_n=v^{j_n}_n,\\
  y_n=J^{-1}\big((1-\lambda_n)JW_n(\bar{v}_n)+\lambda_nJS(\bar{v}_n)\big),\\
  z_n=J^{-1}\big((1-\sigma)J\Pi_CJ^{-1}(J-\eta_nB_1)(y_n)+\sigma J\Pi_CJ^{-1}(J-\eta_nB_2)(y_n)\big),\\
  C_n=\big\lbrace z\in C: \phi(z,\bar{v}_n)\leq \phi(z,x_n)\big\rbrace,\\
  D_n=\big\lbrace z\in E: \phi(z,y_n)\leq \phi(z,\bar{v}_n)\big\rbrace,\\
  M_n=\big\lbrace z\in E: \phi(z,z_n)\leq \phi(z,y_n)\big\rbrace,\\
  x_{n+1}=\Pi_{C_n\cap D_n\cap M_n}(x_1),\quad \forall n\in \mathbb{N},
  \end{cases}
  \end{align}
where $\{\eta_n\}\subseteq(0,+\infty)$, $0<\sigma<1$ and $a,b\in \mathbb{R}$ be such that for all $n\in \mathbb{N}$, $0<a\leq\lambda_n\leq b<1$ and also, $\{\alpha_n\}\subseteq [r, s]\subseteq (0, p)$, where $p=\displaystyle\min_{1\leq j\leq N}\frac{2\tau}{L_j}$ and $\tau$ is given by \eqref{kcv45t}.\\
Then $\{x_n\}$ converges strongly to a point $\omega_0\in \Omega$ where $\omega_0=\Pi_\Omega(x_1)$.
  \end{thm}
  \begin{proof}
    Suppose $h_j(x,y):=\langle y-x,A_j(x)\rangle$ for each $x,y\in C$ and $j=1,2,\cdots,N$. Since $A_j$ is $L_i$-Lipschitz continuous, for every $x, y, z \in C$, we have that
    \begin{align*}
    h_j(x,y)+h_j(y,z)-h_j(x,z)&=\langle y-x,A_j(x)\rangle+\langle z-y,A_j(y)\rangle-\langle z-x,A_j(x)\rangle\\
    &=-\langle y-z,A_j(y)-A_j(x)\rangle\\
    &\geq-\|A_j(y)-A_j(x)\|\|y-z\|\\
    &\geq-L_j\|y-x\|\|y-z\|\\
    &\geq-\frac{L_j}{2}\|y-x\|^2-\frac{L_j}{2}\|y-z\|^2\\
    &\geq-\frac{L_j}{4\tau}\phi(y,x)-\frac{L_j}{4\tau}\phi(z,y).\qquad by \eqref{kcv45t}.
    \end{align*}
    Therefore, $h_j$ is Bregman-Lipschitz-type continuous and $c_{j,1}=c_{j,2}=\frac{L_j}{2\tau}$ with respect to $\frac{1}{2}\|.\|^2$. Moreover, the pseudomonotonicity of $A_j$ ensure the pseudomonotonicity of $h_j$. The conditions A3, A4, and A5 are satisfied automatically. Using Theorem \ref{sdo3mn} and Lemma \ref{fg9jd}, we get the desired result.
  \end{proof}
   \section{\textbf{Numerical Experiment}}
  In the section, we demonstrate our main theorem with an example.
  \begin{ex}
    Let $E=\mathbb{R}, C=[0,1], f(.)=\frac{1}{2}\|.\|$ and $\xi_1=\frac{1}{100}, \xi_2=\frac{2}{100}, \xi_3=\frac{3}{100}$. For $j=1, 2, 3$ define the bifunctions $h_j$ on $C\times C$ into $\mathbb{R}$ as follows:
    \begin{equation*}
      h_j(x,y)=A_j(x)(y-x),
    \end{equation*}
    where
    \begin{align*}
    A_j(x)=
    \begin{cases}
     0, \quad\quad 0\leq x\leq \xi_j,\\
     sin(x-\xi_j), \quad\quad \xi_j\leq x\leq 1.
    \end{cases}
    \end{align*}
    Obviously, the bifunctions $h_j$ satify the conditions A1, A3, A4, and A5. Furthermore
  \begin{align*}
   h_j(x,y)+h_j(y,z)-h_j(x,z)&=(y-z)(A_j(x)-A_j(y))\\
   &\geq-|y-z||x-y|\\
   &\geq-\frac{(y-z)^2}{2}-\frac{(x-y)^2}{2}\\
   &=-D_{\frac{1}{2}\|.\|^2}(z,y)-D_{\frac{1}{2}\|.\|^2}(y,x),
  \end{align*}
  which proves the condition A2 with $c_{j,1}=c_{j,2}=1$. Also, define $W_n=I$, and $S(\bar{v}_n)=\frac{1}{2}\bar{v}_n$, and $\lambda_n=\frac{1}{2}$, and $\alpha_n=\frac{1}{n+2}+\frac{1}{2}$, and $B_1(y_n)=B_2(y_n)=\frac{1}{3}y_n$ for each $n\in \mathbb{N}$. A simple computation shows that algorithm \eqref{mnzkk} takes the following from:
  \begin{align}\label{sjcgv56}
  \begin{cases}
  t^j_n= x_n-\alpha_nA_j(x_n),\quad j=1,2,3,\\
  v^j_n= x_n-\alpha_nA_j(t^j_n),\quad j=1,2,3,\\
  j_n\in Argmax\big\{|v^j_n- x_n|, j=1,2,3\big\},\quad \bar{v}_n=v^{j_n}_n,\\
  y_n=\frac{3}{4}\bar{v}_n,\\
  z_n=\frac{(3-\eta_n)y_n}{3},\\
  C_n=\big\lbrace z\in C: |z-\bar{v}_n|\leq |z-x_n|\big\rbrace,\\
  D_n=\big\lbrace z\in E: |z-y_n|\leq |z-\bar{v}_n|\big\rbrace,\\
  M_n=\big\lbrace z\in E: |z-z_n|\leq |z-y_n|\big\rbrace,\\
  x_{n+1}=P_{C_n\cap D_n\cap M_n}(x_1),\quad \forall n\in \mathbb{N},
  \end{cases}
  \end{align}
  \end{ex}
  It can be observed that all assumptions of Theorem \ref{asli} are satisfied and  $\Omega=\{0\}$. Using algorithm \eqref{sjcgv56} with the initial point $x_1 = 1$, we have $x_n\rightarrow 0$.
  \section{\textbf{Data Availability}}
  No data were used to support the study.

\vspace{0.1in}
\hrule width \hsize \kern 1mm

\begin{thebibliography}{XX}
\bibitem{p20} Ogbuisi FU, Izuchukwu C.: Approximating a zero of sum of two monotone operations which solves a fixed point problem in reflexive Banach spaces. Collectanea Mathematica (in print).
\end{thebibliography}
\end{document}